\documentclass[12pt]{amsart}

\usepackage[latin1]{inputenc}
\usepackage[arrow, matrix, curve]{xy}
\usepackage{amsmath}
\usepackage{ngerman}
\usepackage{amscd}
\usepackage{a4wide}
\usepackage{amsfonts}
\usepackage{amssymb}
\usepackage{amsthm}
\usepackage[T1]{fontenc}
\usepackage[latin1]{inputenc}
\usepackage[english]{babel}

\numberwithin{equation}{section}
\newtheorem{theorem}{Theorem}[section]
\newtheorem{lemma}[theorem]{Lemma}
\newtheorem{keylemma}[theorem]{Key lemma}
\newtheorem{proposition}[theorem]{Proposition}
\newtheorem{corollary}[theorem]{Corollary}

\theoremstyle{definition}

\newtheorem{remark}[theorem]{Remark}

\newtheorem{claim}[theorem]{Claim}
\newtheorem{conj}[theorem]{Conjecture}

\newcommand{\Z}{\mathbb{Z}}

\renewcommand{\O}{\mathcal{O}}
\renewcommand{\r}{\rightarrow}

\title{On a base change conjecture for higher zero-cycles}
\author{Morten L\"uders}
\address{Fakult\"at f\"ur Mathematik, Universit\"at Regensburg, 93040 Regensburg, Germany}
\email{morten.lueders@mathematik.uni-regensburg.de}

\begin{document}
\begin{abstract}
We show the surjectivity of a restriction map on higher $(0,1)$-cycles for a smooth projective scheme over an excellent henselian discrete valuation ring. This gives evidence for a conjecture stated in \cite{KEW} saying that base change holds for such schemes in general for motivic cohomology in degrees $(i,d)$ for fixed $d$ being the relative dimension over the base. Furthermore, the restriction map we study is related to a finiteness conjecture for the $n$-torsion of $CH_0(X)$, where $X$ is a variety over a $p$-adic field.
\end{abstract}

\maketitle
\section{Introduction}
Let $\mathcal{O}_K$ be an excellent henselian discrete valuation ring with quotient field $K$ and residue field $k=\mathcal{O}_K/\pi\mathcal{O}_K$ and always assume that $1/n\in k$.
Let $X$ be a regular scheme, flat and projective over Spec$\mathcal{O}_K$ of fibre dimension $d$. Let $X_K$ denote the generic fibre and $X_0$ the reduced special fibre. Let $\Lambda=\mathbb{Z}/n\mathbb{Z}$.

In \cite[Cor. 9.5]{SS} and \cite[App.]{EWB} it is shown that for $X\rightarrow \text{Spec}\mathcal{O}_K$ smooth and projective and $k$ finite or algebraically closed, the restriction map
$$CH_1(X)_{\Lambda}\xrightarrow{\simeq} CH_0(X_0)_{\Lambda}$$
is an isomorphism of Chow groups with coefficients in $\Lambda$. This result is reproven in \cite{KEW} for more general residue fields and generalised to the case that $X_0$ is a simple normal crossings divisor. In that case one needs to replace $CH_0(X_0)$ by $H^{2d}_{cdh}(X_0,\mathbb{Z}/n\mathbb{Z}(d))$, i.e. the hypercohomology of the motivic complex $\mathbb{Z}/n\mathbb{Z}(d)$ in the cdh-topology, which is isomorphic to $CH_0(X_0)$ for $X_0/k$ smooth. The result then says that if $k$ is finite, or algebraically closed, or $(d-1)!$ prime to $m$, or $A$ is of equal characteristic, or $X/\mathcal{O}_K$ is smooth with perfect residue field $k$, then there is an isomorphism
$$CH_1(X)_{\Lambda}\xrightarrow{\simeq} H^{2d}_{cdh}(X_0,\mathbb{Z}/n\mathbb{Z}(d))$$
which is induced by restricting a one-cycle in general position to a zero-cycle on $X_0^{sm}$. 
Generalising this result, the following conjecture is stated in section $10$ of \cite{KEW}: 
\begin{conj}\label{conjkew} The restriction homomorphism 
$$res:H^{i,d}(X,\mathbb{Z}/n\mathbb{Z})\rightarrow H^{i,d}_{cdh}(X_0,\mathbb{Z}/n\mathbb{Z})$$
is an isomorphism for all $i\geq 0$. \end{conj} 
Here $H^{i,d}(X,\mathbb{Z}/m\mathbb{Z})=H^{i}(X,\mathbb{Z}/m\mathbb{Z}(d))$ are the motivic cohomology groups for schemes over Dedekind rings defined in \cite{Sp}. In this article we consider the corresponding restriction map on higher Chow groups of zero-cycles with coefficients in $\Lambda$
$$res^{CH}:CH^d(X,2d-i)_{\Lambda}\xrightarrow{} CH^d(X_0,2d-i)_{\Lambda}$$
for $X/\mathcal{O}_K$ smooth which we define to be induced by the following composition:
$$res^{CH}:CH^n(X,m)\xrightarrow{}CH^n(X_K,m)\xrightarrow{\cdot(-\pi)}CH^{n+1}(X_K,m+1)\xrightarrow{\partial}CH^n(X_0,m).$$ 
Here $\cdot(-\pi)$ is the product with $-\pi\in CH^1(K,1)=K^{\times}$ defined in \cite[Sec. 5]{Bl}, $\pi$ is a local parameter for the discrete valuation on $K$ and $\partial$ is the boundary map coming from the localization sequence for higher Chow groups (see \cite{Le2}). We call the composition
$$sp_{\pi}^{CH}:CH^n(X_K,m)\xrightarrow{\cdot(-\pi)}CH^{n+1}(X_K,m+1)\xrightarrow{\partial}CH^n(X_0,m)$$
a specialisation map for higher Chow groups. We note that $res^{CH}$ does not depend on the choice of $\pi$ whereas $sp_{\pi}^{CH}$ does. For a detailed discussion of the specialisation map see also \cite[Sec. 3]{ADIKMP}.

Our main theorem is the following:
\begin{theorem}\label{mt}
Let $X/\mathcal{O}_K$ be smooth. Then the restriction map $$res^{CH}:CH^d(X,1)_{\Lambda}\twoheadrightarrow CH^d(X_0,1)_{\Lambda}$$ is surjective. This implies in particular the surjectivity part of conjecture \ref{conjkew} for the pair $(2d-1,d)$.
\end{theorem} 
This implies the following corollary:
\begin{corollary}\label{cormt}
Let $X/\mathcal{O}_K$ be smooth. Then the specialisation map $$sp^{CH}_{\pi}:CH^d(X_K,1)_{\Lambda}\twoheadrightarrow CH^d(X_0,1)_{\Lambda}$$ is surjective. 
\end{corollary} 
The restriction map in the degree of theorem \ref{mt} is of particular interest since it is related to a conjecture on the finiteness of $CH^d(X_K)[n]$ for $K$ a $p$-adic field. This is shown in section \ref{remarks} as well as the injectivity for $d=2$. Furthermore theorem \ref{mt} together with the main result of \cite{KEW} may be considered as a generalization to perfect residue fields of the vanishing of the Kato homology group $KH_3(X)$ defined in \cite{SS} where it was proven for $k$ finite or separably closed. \\

\paragraph{\textit{Acknowledgements.}} I would like to heartily thank my supervisor Moritz Kerz for his help and many ideas for this paper. I would also like to thank Johann Haas and Yigeng Zhao for helpful discussions and the department of mathematics of the university of Regensburg and the SFB 1085 "Higher Invariants" for the nice working environment.

\section{Main result}\label{mainresult}
Let $\mathcal{O}_K$ be an excellent henselian discrete valuation ring with quotient field $K$ and residue field $k=\mathcal{O}_K/\pi\mathcal{O}_K$ and always assume that $1/n\in k$.
From now on let $X$ be a smooth and projective scheme over Spec$\mathcal{O}_K$ of fibre dimension $d$ in which case we also say that $X$ is of relative dimension $d$ over $\O_K$. Let $X_K$ denote the generic fibre and $X_0$ the reduced special fibre. By $X_{(p)}$ we denote the set of points $x\in X$ such that $\text{dim}(\overline{\{x\}})=p$, where $\overline{\{x\}}$ denotes the closure of $x$ in $X$.

We are going to use the following notation for Rost's Chow groups with coefficients in Milnor K-theory (see \cite[Sec. 5]{Ro}):
$$C_p(X,m)=\bigoplus_{x\in X_{(p)}}(K_{m+p}^Mk(x))\otimes \mathbb{Z}/n\Z$$
$$Z_p(X,m)=\text{ker}[\partial:C_p(X,m)\rightarrow C_{p-1}(X,m)]$$
$$A_p(X,m)=H_p(C_*(X,m))$$
We write $Z_k(X)$ for the group of $k$-cycles on $X$, i.e. the free abelian group generated by $k$-dimensional closed subschemes of $X$.

Let $\pi$ be some fixed a local parameter of $\O_K$. We define the restriction map
$$res_\pi: C_p(X,m)\rightarrow C_{p-1}(X_0,m+1)$$
to be the composition
$$res_{\pi}: C_p(X,m)\rightarrow C_{p-1}(X_K,m+1)\xrightarrow{\cdot\{-\pi\}} C_{p-1}(X_K,m+2)\xrightarrow{\partial}C_{p-1}(X_0,m+1).$$
In the above composition the map $C_p(X,m)\rightarrow C_{p-1}(X_K,m)$ is defined to be the identity on all elements supported on $X_{(p)}\setminus {X_0}_{(p)}$ and zero on ${X_0}_{(p)}$. The map $\partial$ is defined to be the boundary map induced by the tame symbol on Milnor K-theory for discrete valuation rings.
More precisely, $\partial$ is defined as follows: Let $\overline{\{x\}}$ be the subscheme corresponding to $x\in X_{(p)}$. Let us assume for simplicity that $\overline{\{x\}}$ is normal. Otherwise we take the normalisation and use the norm map. Now if $y\in \overline{\{x\}}_{(p-1)}$, then $y$ defines a discrete valuation on $k(x)$. Let $\pi'$ be a local parameter of $k(x)$. Let $\partial^x_y:K^M_{n+1}k(x)\r K^M_{n}k(y)$ be the tame symbol defined by sending $\{\pi',u_1,...,u_n\}$ to $\{\bar{u_1},...,\bar{u_n}\},$ where the $u_i$ are units in the discrete valuation ring of $k(x)$ and the $\bar{u_i}$ their images in $k(y)$. $\partial$ is defined to be the sum of all $\partial^x_y$ taken over all $x\in X_{(p)}$ and all $y\in \overline{\{x\}}_{(p-1)}$. Note that the restriction map $res_\pi$ has to be distinguished from the specialisation map
$$sp^x_{y,\pi'}=\partial^x_y\circ \{-\pi'\}:K_{n}^Mk(x)\rightarrow K_{n}^Mk(y).$$
$sp^x_{y,\pi'}$ sends $\{\pi'^{i_1}u_1,...,\pi'^{i_n}u_n\}$ to $\{\bar{u_1},...,\bar{u_n}\},$ where again the $u_i$ are units in the discrete valuation ring of $k(x)$ and the $\bar{u_i}$ their images in $k(y)$.  
 
The map $res_\pi$ depends on the choice of $\pi$ but the induced map on homology
$$res: A_p(X,m)\rightarrow A_{p-1}(X_0,m+1)$$
is independent of the choice. This can be seen as follows: Let $u\in\O_K^\times$ and $\alpha\in C_p(X,m)$. Then $res_{u\pi}(\alpha)=\partial(\{-\pi u\}\cdot \alpha)=\partial(\{-\pi\}\cdot \alpha)+\partial(\{u\}\cdot \alpha)=res_{\pi}(\alpha)+\partial(\{u\}\cdot \alpha)$. Now if $\alpha\in A_p(X,m)$, then $\partial(\{u\}\cdot \alpha)=0$ and $res_{u\pi}(\alpha)=res_{\pi}(\alpha)$. In the following we will write $res$ for $res_\pi$, fixing a local parameter $\pi\in O_K$.

We now turn to our principle interest of study, the restriction map 
$$res:C_2(X,-1)\rightarrow C_1(X_0,0).$$

We start with the following lemma:
\begin{lemma}\label{Csurj} The map $res:C_2(X,-1)\rightarrow C_1(X_0,0)$, after having fixed $\pi$, is surjective. \end{lemma}
\begin{proof}
Let $\bar{u}\in K^M_1k(x)$ for some $x\in X_0^{(d-1)}$. As in the proof of \cite[Lem. 7.2]{SS} we can find a relative surface $Z\subset X$ containing $x$ and being regular at $x$ and such that $Z\cap X_0$ contains $\overline{\{x\}}$ with multiplicity $1$. Let $Z_0=\cup_{i\in I} Z_0^{(i)}\cup \overline{\{x\}}$ be the union of the pairwise different irreducible components of the special fiber of $Z$ with those irreducible components different from $\overline{\{x\}}$ indexed by $I$. 
Since all maximal ideals, $m_i$ corresponding to $Z_0^{(i)}$ and $m_x$ corresponding to $\overline{\{x\}}$, in the semi-local ring $\O_{Z,Z_0}$ are coprime, the map $\O_{Z,Z_0}\r \prod_{i\in I}\O_{Z,Z_0}/m_i\times \O_{Z,Z_0}/m_x$ is surjective. 
Therefore we can find a lift $u\in K^M_1k(z)$, $z$ being the generic point of $Z$, of $\bar{u}$ which specialises to $\bar{u}$ in $K(\overline{\{x\}})^{\times}$ and to $1$ in $K(Z_0^{(i)})^{\times}$ for all $i\in I$.
\end{proof}

The main result we are going to prove is the following:
\begin{proposition}\label{Asurj} The restriction map $res:A_2(X,-1)\rightarrow A_1(X_0,0)$ is surjective. \end{proposition}
It will be implied by the following key lemma:
\begin{keylemma}\label{keylemma1} Let $\xi\in \text{ker}[Z_1(X)/n\overset{res}{\rightarrow} Z_0(X_0)/n]$, then there is a $\xi'\in \text{ker}[C_2(X,-1)\overset{res}{\rightarrow} C_1(X_0,0)]$ with $\partial(\xi')=\xi$.
\end{keylemma}
\begin{proof}(Proposition \ref{Asurj}) Let $\xi_0\in \text{ker}[C_1(X_0,0)\overset{\partial}{\rightarrow}C_0(X_0,0)]$. By lemma \ref{Csurj} there is a $\xi\in C_2(X,-1)$ with $res(\xi)=\xi_0$. As $res(\partial (\xi))=\partial(res(\xi))=0$, key lemma \ref{keylemma1} tells us that there is a $\xi'\in \text{ker}(C_2(X,-1)\rightarrow C_1(X_0,0))$ with $\partial\xi'=\partial\xi$. As $res$ is a homomorphism, it follows that $\xi_0= res(\xi-\xi')$ and $\partial(\xi-\xi')=0$. Hence $res:Z_2(X,-1)\rightarrow Z_1(X_0,0)$ is surjective and the commutativity of $\partial$ and $res$ implies that $res:A_2(X,-1)\rightarrow A_1(X_0,0)$ is surjective. 
\end{proof}

\begin{proof} (Key lemma \ref{keylemma1}) We start with the case of relative dimension $d=1$, i.e. $X$ is a smooth fibered surface over $\mathcal{O}_K$, and consider the following diagram:
$$\begin{xy} 
  \xymatrix{
     C_2(X,-1)=K(X)^*\otimes \mathbb{Z}/n\Z \ar[r]^{res} \ar[d]_{\partial} &  C_1(X_0,0)=K(X_0)^*\otimes \mathbb{Z}/n\Z  \ar[d]^{\partial}  \\
      Z_1(X)/n \ar[r]^{res}     & Z_0(X_0)/n  
  }
\end{xy} $$
where we write $Z_i(X)/n$ for $C_{i}(X,-i)$ which are just the cycles of dimension $i$ modulo $n$. The restriction map in the lowest degree $res:Z_1(X)/n \r Z_0(X_0)/n$ agrees with the specialisation map on cycles defined by Fulton in \cite[Rem. 2.3]{Fu} since $X_0$ is a principle Cartier divisor and $\partial^x_y(\{-\pi\})=\text{ord}_{\O_{\overline{\{x\}},y}}(\pi)$. Modifying $\xi\in \text{ker}[Z_1(X)/n\overset{res}{\rightarrow} Z_0(X_0)/n]$ by elements equivalent to zero in $Z_1(X)/n$, we may represent it by an element $x\in\text{ker}[Z_1(X)\rightarrow Z_0(X_0)]$. 

We consider the following short exact sequence of sheaves:
\begin{equation}\label{eq1}
0\rightarrow \mathcal{O}^*_{X;X_0}\rightarrow \mathcal{M}^*_{X;X_0}\rightarrow Div(X,X_0)\rightarrow 0,
\end{equation}
where $\mathcal{M}^*_{X;X_0} (\text{resp. }\mathcal{O}^*_{X;X_0})$ denotes the sheaf of invertible meromorphic functions (resp. invertible regular functions) relative to $\text{Spec}\mathcal{O}_K$ and congruent to $1$ in the generic point of $X_0$, i.e. in $\mathcal{O}_{X,{\mu}}$, where $\mu$ is the generic point of $X_0$, and $Div(X,X_0)$ is the sheaf assoiciated to $\mathcal{M}^*_{X;X_0}/\mathcal{O}^*_{X;X_0}$. In other words, $Div(X,X_0)(U)$ is the set of relative Cartier divisors on $U\subset X$ which specialise to zero in $X_0$. For the concept of relative meromorphic funtions and divisors see \cite[Sec. 20, 21.15]{EGA4}.

We want to show that $(Div(X,X_0)(X)/\mathcal{M}^*_{X;X_0}(X))/n=0$.
\begin{claim} $\text{Pic}(X,X_0)\cong Div(X,X_0)(X)/\mathcal{M}^*_{X;X_0}(X)$.
\end{claim}
Short exact sequence (\ref{eq1}) induces the following exact sequence:
$$\mathcal{O}^*_{X;X_0}(X)\rightarrow\mathcal{M}^*_{X;X_0}(X)\rightarrow Div(X,X_0)(X)\rightarrow \text{Pic}(X,X_0)\rightarrow H^1(X,\mathcal{M}^*_{X;X_0})$$
Now $\text{Pic}(X,X_0)=H^1(X,\mathcal{O}^*_{X;X_0})$ can also be described as the group of isomorphism classes of pairs $(\mathcal{L},\psi)$ of an invertible sheaf $\mathcal{L}$ with a trivialisation $\psi:\mathcal{L}|_{X_0}\cong \mathcal{O}_{X_0}$ (see f.e. \cite[Lem. 2.1]{SV}). 

The following argument shows that the map $Div(X,X_0)(X)\rightarrow \text{Pic}(X,X_0)$ is surjective: Let $(\mathcal{L},\psi)\in \text{Pic}(X,X_0)$. The trivialisation $\psi$ gives an isomorphism $\psi:\mathcal{L}\otimes_{\mathcal{O}_X}\mathcal{O}_{X_0}\xrightarrow{\cong}\mathcal{O}_{X_0}$ and by localising an isomorphism $\psi_\mu:\mathcal{L}_{\mu}\otimes_{\mathcal{O}_{X,{\mu}}}\mathcal{O}_{X_0,{\mu}}\xrightarrow{\cong}\mathcal{O}_{X_0,{\mu}}$, where $\mu$ again denotes the generic point of $X_0$. Let $s$ denote a lift of $\psi_{\mu}^{-1}(1)$ under the surjective map $\mathcal{L}_{\mu}\twoheadrightarrow \mathcal{L}_{\mu}\otimes_{\mathcal{O}_{X,{\mu}}}\mathcal{O}_{X_0,{\mu}}$. Then $s$ is a meromorphic section of $\mathcal{L}$ and the divisor div$(s)\in Div(X,X_0)(X)$ maps to $(\mathcal{L},\psi)$.  

It follows that $\text{Pic}(X,X_0)\cong Div(X,X_0)(X)/\mathcal{M}^*_{X;X_0}(X)$. \qed 
\begin{claim} $\text{Pic}(X,X_0)$ is uniquely $n$-divisible. \end{claim} 
Since $$\text{Pic}(X,X_0)\cong \varprojlim_m\text{Pic}(X_m,X_0)\cong\varprojlim_m H^1(X_0,1+\pi\mathcal{O}_{X_m}),$$ where the first isomorphism follows from \cite[Thm. 5.1.4]{EGA3}, it suffices to show that $H^1(X_0,1+\pi\mathcal{O}_{X_m})$ is uniquely $n$-divisible. This can be seen as follows: 
$$1+\pi\mathcal{O}_{X_m}\supset 1+\pi^{2}\mathcal{O}_{X_m}\supset ...\supset 1$$
defines a finite filtration on the sheaf $1+\pi\mathcal{O}_{X_m}$ with graded pieces $gr^n=(\pi)^n/(\pi)^{n+1}\cong \mathcal{O}_{X_0}\otimes (\pi)^n$. We use this filtration to define a filtration on $H^1(X_0,1+\pi\mathcal{O}_{X_m})$ by 
$$F^n:=\text{Im}(H^1(X_0,1+\pi^n\mathcal{O}_{X_m})\rightarrow H^1(X_0,1+\pi\mathcal{O}_{X_m})).$$
The unique divisibility of $H^1(X_0,1+\pi\mathcal{O}_{X_m})$ follows now by descending induction from the exact sequence 
$$0\rightarrow 1+\pi^{n+1}\mathcal{O}_{X_m}\rightarrow 1+\pi^{n}\mathcal{O}_{X_m}\rightarrow gr^n\rightarrow 0,$$
the unique divisibility of $H^i(X_0,\mathcal{O}_{X_0}\otimes \pi^n)$ as a finitely generated $k$-module and the five-lemma.  \end{proof}

It follows that $\text{Pic}(X,X_0)/n\cong (Div(X,X_0)(X)/\mathcal{M}^*_{X;X_0}(X))/n=0$ and therefore that the class of $x$ in $Z_1(X)/n$, i.e. $\xi$, is in the image of $\text{ker}[C_2(X,-1)\overset{res}{\rightarrow} C_1(X_0,0)]$ under $\partial$.  

We now do the induction step for $X$ of arbitraty relative dimension $d>1$ over Spec$\mathcal{O}_K$, assuming that the key lemma holds for relative dimension $d-1$, using an idea of Bloch put forward in \cite[App.]{EWB}. By a standard norm argument we may from now on assume that $k$ is infinite. 

As above we may represent $\xi$ by an element of $\text{ker}[Z_1(X)\rightarrow Z_0(X_0)]$ and as in the proof of \cite[Prop. 4.1]{KEW} we may assume that $\xi$ is represented by a cycle of the form $[x]-r[y]\in \text{ker}[Z_1(X)\rightarrow Z_0(X_0)]$ with $x$ and $y$ integral and such that $y$ is regular and has intersection number $1$ with $X_0$. Let us recall the argument: First note that one can lift a reduced closed point of $X_0$ to an integral horizontal one-cycle having intersection number $1$ with $X_0$. Now if $\xi=\sum_{i=1}^s n_i[x_i]\in \text{ker}[Z_1(X)\rightarrow Z_0(X_0)]$, then we lift $(x_i\cap X_0)_{\text{red}}$ to a one-cycle $y_i$ of the aforementioned type. Furthermore, we choose the same $y_i$ for all the $x_i$ intersecting $X_0$ in the same closed point. Let $r_i$ be the intersection multiplicity of $x_i$ with $X_0$. Then also $\sum_{i=1}^s n_ir_i[y_i]\in \text{ker}[Z_1(X)\rightarrow Z_0(X_0)]$ and it suffices to show the statement for each $x_i-r_iy_i$ separately, i.e. the claim follows. 

Let $\tilde x$ be the normalisation of $x$. Since $\O_K$ is excellent, $\tilde x$ is finite over $x$. This implies that there is an imbedding $\tilde x\hookrightarrow X':=X\times_{\text{Spec}\mathcal{O}_K}\mathbb{P}^N$ such that the following diagram commutes:
$$\begin{xy} 
  \xymatrix{
  \tilde x \ar[r]^{} \ar[d]_{} & X'=X\times_{\text{Spec}\mathcal{O}_K}\mathbb{P}^N  \ar[d]^{pr_X}  \\
     x \ar[r]^{} \ar[d]_{} & X \ar[d]^{}  \\
     \text{Spec}\mathcal{O}_K \ar[r]^{=}     &  \text{Spec}\mathcal{O}_K
  }
\end{xy} $$
Let $[\tilde x\cap X_0']= r'[\bar z]$ for $\bar z$ an integral zero-dimensional subscheme of $X'_0$. We take a regular lift $z$ of $\bar z$ in $y\times \mathbb{P}^N\subset X'$ which has intersection number $1$ with $X_0'$ and get that $[\tilde x]-r'[z]\in ker[Z_1(X')\rightarrow Z_0(X_0')]$ and $pr_{X*}([\tilde x]-r'[z])=[x]-r[y]=\xi$. 

We now use a Bertini theorem by Altman and Kleiman to prove key lemma \ref{keylemma1} by an induction on the relative dimension of $X$ over $\O_K$.
\begin{lemma}\label{subschemes} There exist smooth closed subschemes $Z,Z'\subset X'$ with the following properties:
\begin{enumerate}
         \item $Z$ has fiber dimension one, $Z'$ has fiber dimension $d-1$.
         \item $Z$ contains $\tilde x$, $Z'$ contains $z$.
         \item The intersection $Z\cap Z'\cap X'_0$ consist of reduced points.
      \end{enumerate}  
\end{lemma}
\begin{proof}
First note that for a sheaf of ideals $\mathcal{J}\subset \mathcal{O}_{X'}$  we have the following short exact sequence:
$$0\rightarrow \mathcal{J}\otimes_{\mathcal{O}_{X'}}\mathcal{O}_{X'}(-[X_0'])(M)\rightarrow \mathcal{J}\otimes_{\mathcal{O}_{X'}}\mathcal{O}_{X'}(M)\rightarrow \mathcal{J}\otimes_{\mathcal{O}_{X'}}i_*\mathcal{O}_{X_0'}(M)\rightarrow 0$$
for $i:X'_0\hookrightarrow X'$ and $M\in\mathbb{Z}$. For $M\gg 0$ Serre vanishing implies that $H^1(X', \mathcal{F}(M))=0$ for $\mathcal{F}$ coherent and therefore that the map
$$\Gamma(\mathcal{J}\otimes_{\mathcal{O}_{X'}}\mathcal{O}_{X'}(M))\twoheadrightarrow \Gamma(\mathcal{J}\otimes_{\mathcal{O}_{X'}}\mathcal{O}_{X'_0}(M))$$ 
is surjective. This allows us to lift the sections on the right defining subvarieties of $X_0'$ to sections of a twisted sheaf of ideals on $X'$. 

Let $\mathcal{J}_{\tilde x}$ be the sheaf of ideals defining $\tilde x$ and $\mathcal{J}_z$ be the sheaf of ideals defining $z$. Let $p\in\tilde x\cap X'_0$ $(q\in z\cap X'_0)$. Then $\text{dim}_{X_0}(p)=d\geq 2$ and since $\tilde x$ (resp. $z$) is regular, we have that $e_{\tilde x\cap X_0'}(p)\leq e_{\tilde x}(p)=\text{dim}_{k(p)}(\Omega^1_{\tilde x}(p))=1<2$, where $e_{\tilde x}(p)$ is the embedding dimension of $\tilde x$ at $p$ and analogously for $q$. Therefore by \cite[Thm. 7]{AK} we can find sections in $\bar\sigma_1,...,\bar\sigma_{d+N-1}\in\mathcal{J}_{\tilde x}|_{X_0'}(M)$ (resp. $\bar\sigma'\in\mathcal{J}_{\tilde x}|_{X_0'}(M)$) defining smooth subschemes containing $p$ (resp. $q$) that intersect transversally. Let $\sigma_1,...,\sigma_{d+N-1}$ (resp. $\sigma'$) be liftings under the surjections $\Gamma(\mathcal{J}_{\tilde{x}}\otimes_{\mathcal{O}_{X'}}\mathcal{O}_{X'}(M))\twoheadrightarrow \Gamma(\mathcal{J}_{\tilde{x}}\otimes_{\mathcal{O}_{X'}}\mathcal{O}_{X'_0}(M))$ and $\Gamma(\mathcal{J}_z\otimes_{\mathcal{O}_{X'}}\mathcal{O}_{X'}(M))\twoheadrightarrow \Gamma(\mathcal{J}_z\otimes_{\mathcal{O}_{X'}}\mathcal{O}_{X'_0}(M))$. Then the complete intersections $Z:=V(\sigma_1,...,\sigma_{d+N-1})$ and $Z':=V(\sigma')$ have the desired properties. 
\end{proof}

Using these subschemes, we can now do the induction step and finish the proof of the key lemma. Since $Z\cap Z'\cap X_0'$ consists of reduced points, the component $z'$ of $Z\cap Z'$ that contains $z\cap X_0'$ has intersection number $1$ with $X_0'$ and is a regular curve as it is regular over the closed point of Spec$\mathcal{O}_K$. Now since $Z'$ is of relative dimension $d-1$ and $z$ and $z'$ both lie in $Z'$ and satisfy $res([z']-[z])=0$, we get by the induction assumption that there is a $\xi$ with support on $Z'$ restricting to $1$ and with $\partial(\xi)=[z']-[z]$. 

By the relativ dimension one case proved in the beginning we get that for $\tilde x,z'\subset Z$ and $[\tilde x]-r'[z']$, which also restricts to $0$, there is a $\xi'$ with support on $Z$ such that $res(\xi')=0$ and $\partial(\xi')=[\tilde x]-r'[z']$. It follows that $res(\xi'+r\xi)=1$ and $\partial(\xi'+r\xi)=[\tilde x]-r'[z]$. 

By the commutativity of the following diagram we get the result.
$$\begin{xy} 
  \xymatrix{
   & C_2(X',-1) \ar[d]_{} \ar[ddl]_{} \ar[rr] & & C_1(X_0',0) \ar[d]_{} \ar[ddl]_{} \\
   & Z_1(X')/n \ar[ddl]_{}  \ar[rr] & & Z_0(X_0')/n \ar[ddl]_{} \\
   C_2(X,-1) \ar[d]_{} \ar[rr]^{} & & C_1(X_0,0) \ar[d]_{} &   \\
    Z_1(X)/n \ar[rr]^{} & & Z_0(X_0)/n 
  }
\end{xy} $$
The commutativity of the diagram follows from \cite[Sec. 4]{Ro} since all the maps in question are defined in terms of the 'four basic maps' which are compatible. \qed

\begin{corollary} The restriction map $$res^{CH}:CH^d(X,1)_{\Lambda}\rightarrow CH^d(X_0,1)_{\Lambda}$$ defined in the introduction is surjective.  \end{corollary}
\begin{proof} We first show that the homology of the sequence
$$\oplus_{x\in X_0^{(d-2)}}K_{2}^Mk(x)\rightarrow \oplus_{x\in X_0^{(d-1)}}K_{1}^Mk(x)\rightarrow \oplus_{x\in X_0^{(d)}}K_{0}^Mk(x)$$
is isomorphic to $CH^d(X_0,1)$ which implies that $A_1(X_0,0)\cong CH^d(X_0,1)_{\Lambda}$. This follows from the spectral sequence 
\begin{equation}\label{niveauspectralseq}
E^{p,q}_1=\oplus_{x\in X_0^{(p)}}CH^{r-p}(\text{Spec}k(x),-p-q)\Rightarrow CH^r(X_0,-p-q)
\end{equation} 
(see \cite[Sec. 10]{Bl}) for $r=d=\text{dim} X_0$, the fact that $CH^r(k(x),r)\cong K^M_r(k(x))$ and the vanishing of $CH^{r}(\text{Spec}k(x),j)$ for $r>j$.

Using a limit argument and the localization sequence for schemes over a regular noetherian base $B$ of dimension one constructed in \cite{Le2}, we also get the existence of spectral sequence (\ref{niveauspectralseq}) for $X/\mathcal{O}_K$. Now for the same reasons as above this spectral sequence implies that the homology of
$$\oplus_{x\in X^{(d-2)}}K_{2}^Mk(x)\rightarrow \oplus_{x\in X^{(d-1)}}K_{1}^Mk(x)\rightarrow \oplus_{x\in X^{(d)}}K_{0}^Mk(x)$$
is isomorphic to $CH^d(X,1)$ which implies that $A_2(X,-1)\cong CH^d(X,1)_{\Lambda}$. 

The result now follows from proposition \ref{Asurj} and the compatibility of $res$ and $res^{CH}$. 
\end{proof}

\begin{remark}
The isomorphism $A_1(X_0,0)\cong CH^d(X_0,1)_{\Lambda}$ also follows from the isomphism $CH(X,1)\cong H^{p-1}(X,\mathcal{K}_p)$ for $p\geq 0$ and $\mathcal{K}_p$ the K-theory sheaf (see f.e. \cite[Cor. 5.3]{M}).
\end{remark}

\section{Remarks on the injectivity of res}\label{remarks}
In this section we prove the injectivity of the restriction map for $d=2$ in our case and remark on imlications of the conjectured injectivity.
\begin{conj}\label{propinj} The map $res:A_2(X,-1)\rightarrow A_1(X_0,0)$ is injective. \end{conj}
\begin{proposition}\label{propinj2} Conjecture \ref{propinj} holds for $X/\mathcal{O}_K$ of relative dimension $2$. 
\end{proposition}
\begin{proof}
Let $\Lambda:=\mathbb{Z}/n$ and $\Lambda(q):=\mu_n^{\otimes q}$. We use the coniveau spectral sequence
$$E_1^{p,q}(X,\Lambda(c))=\coprod_{x\in X^p}H^{p+q}_x(X,\Lambda(c))\Rightarrow H_{\text{\'et}}^{p+q}(X,\Lambda(c)),$$ where $H_x^*$ is \'etale cohomology with support in $x$. 

Cohomological purity  (respectively absolute purity) gives isomorphisms $H^{p+q}_x(X,\Lambda(c))\cong H^{q-p}(k(x),\Lambda(c-p))$ which lets us write the above spectral sequence in the following form:
$$E_1^{p,q}(X,\Lambda(c))=\coprod_{x\in X^p}H^{q-p}(k(x),\Lambda(c-p))\Rightarrow H_{\text{\'et}}^{p+q}(X,\Lambda(c)).$$
For more details see for example \cite{CHK}. Writing out this spectral sequence for $X$ and $X_0$ respectively and using the norm residue isomorphism $K^M_n(k)/m\cong H^n(k,\mu_m^{\otimes n})$ for $n\leq 2$ (see \cite{MS}), we get injective edge morphisms $A_2(X,-1)\hookrightarrow H^3_{\text{\'et}}(X,\Lambda(2))$ and $A_1(X_0,-1)\hookrightarrow H^3_{\text{\'et}}(X_0,\Lambda(2))$ for dimensional reasons. The restriction map induces a map between these spectral sequences and therefore a commutative diagram

$$\begin{xy} 
  \xymatrix{
  A_2(X,-1) \ar[r]^{} \ar@{^{(}->}[d]_{} & A_1(X_0,0)  \ar@{^{(}->}[d]^{}  \\
     H^3_{\text{\'et}}(X,\Lambda(2)) \ar[r]^{\cong}  & H^3_{\text{\'et}}(X_0,\Lambda(2))  
  }
\end{xy} $$ 
whose lower horizontal morphism is an isomorphism by proper base change. It follows that $A_2(X,-1)\rightarrow A_1(X_0,0)$ is injective. 
\end{proof}

\begin{remark}\label{remarkinj} The injectivity of $res$ would have implications for a finiteness conjecture on the $n$-torsion of $CH_0(X_K)$ for $X_K$ a smooth scheme over a $p$-adic field with finite residue field and good reduction (see for example \cite{Co}). More precisely, using the coniveau spectral sequence, we can see that the group $A_1(X_K,0)$ is isomorphic to $H^{2d-1}_{Zar}(X_K,\mathbb{Z}/n(d))$ and therefore surjects onto $CH_0(X_K)[n]$. Furthermore it fits into the exact sequence (see \cite[Sec. 5]{Ro})
$$A_2(X,-1)\rightarrow A_1(X_K,0)\rightarrow A_1(X_0,-1)\cong CH_1(X_0)/n.$$
Now conjecture \ref{propinj} implies that there is a sequence of injections $A_2(X,-1)\hookrightarrow A_1(X_0,0)\hookrightarrow H^{2d-1}_{\text{\'et}}(X_0,\mathbb{Z}/n(d))$ into the finite group $H^{2d-1}_{\text{\'et}}(X_0,\mathbb{Z}/n(d))$. Note that the second injection follows from the Kato conjectures. More precisely, there is an exact sequence 
$$KH_3(X_0,\Z/n\Z)\r A_1(X_0,0)\cong CH^d(X_0,1)_{\Lambda}\r H^{2d-1}_{\text{\'et}}(X_0,\mathbb{Z}/n(d))$$
(see \cite[Lem. 6.2]{JS}) and the Kato homology group $KH_3(X_0,\Z/n\Z)$ is zero due to the Kato concectures (see \cite{KS}).
Therefore the finiteness of $CH_0(X_K)[n]$ would depend on the finiteness of $CH_1(X_0)/n.$

In the case of relative dimension $2$ the finiteness of $CH_1(X_0)/n\cong \text{Pic}(X_0)/n$ can be shown using the injection $\text{Pic}(X_0)/n\hookrightarrow H^2_{\text{\'et}}(X_0,\mu_n)$ and the finiteness of $H^2_{\text{\'et}}(X_0,\mu_n)$ (see f.e. \cite[VI.2.8]{Mi}). Therefore proposition \ref{propinj2} implies in particular the finiteness of $CH_0(X_K)[n]$ for $X_K$ a smooth surface over a $p$-adic field with finite residue field and good reduction which is a well-known result by Bloch (see f.e. \cite[Thm. 3.3.2]{Co2}).
\end{remark}
\begin{remark} In the light of remark \ref{remarkinj} and the base change conjecture for higher zero-cycles stated in the introduction one might ask if $$CH^d(X_K,i)[n]$$ is finite for all $i\geq 0$ for smooth schemes over $p$-adic fields.
\end{remark}

\bibliographystyle{plain}

\end{document}